\documentclass{amsart}
\usepackage{pb-diagram}
\usepackage{verbatim}
\usepackage[leqno]{amsmath}
\usepackage{amssymb,amsthm,pb-diagram,pb-xy}
\usepackage[cmtip, arrow]{xy}
\usepackage{amscd}
\usepackage[mathscr]{euscript}
\usepackage[numbers]{natbib}
\newtheorem{teor}{Theorem}[section]
\newtheorem{lema}[teor]{Lemma}

\theoremstyle{definition}
\newtheorem{defi}[teor]{Definition}
\newtheorem{prop}[teor]{Proposition}
\newtheorem{coro}[teor]{Corollary}
\newtheorem{example}[teor]{Example}

\newtheorem{fact}{Fact}

\theoremstyle{remark}
\newtheorem{nota}[teor]{Remark}

\numberwithin{equation}{section}

\newcommand{\C}{{\ensuremath{\mathbb{C}}}}
\newcommand{\N}{{\ensuremath{\mathbb{N}}}}
\newcommand{\Z}{{\ensuremath{\mathbb{Z}}}}
\newcommand{\Q}{{\ensuremath{\mathbb{Q}}}}

\newcommand{\F}{\mathbb{F}}

\def\D{D^{\omega }(G)}
\begin{document}

\title{Twisted K-theory for the Orbifold $[*/G]$}

\author{Mario Vel\'asquez}
\address{Department of Mathematics, Universidad de Los Andes, Bogot\'a, Colombia}
\curraddr{Universidad de Los Andes, Departamento de Matem\'aticas, Cra 1 No 18A- 12 edificio H, Bogot\'a, Colombia}
\email{ma.velasquez109@uniandes.edu.co}
\thanks{The first author was partially supported by COLCIENCIAS through the grant `Becas Generaci\'on del Bicentenario' \#494,  FUNDACI\'ON MAZDA PARA EL ARTE Y LA CIENCIA. and by Prof. Wolfgang L\"uck through his `Leibniz Prize'.}
\author{Edward Becerra}
\address{Department of Mathematics, Universidad Nacional de Colombia,
Bogot\'a, Colombia}
\email{esbecerrar@unal.edu.co}
\author{Hermes Martinez}
\address{Department of Mathematics, Universidad Sergio Arboleda, Bogot\'a, Colombia}
\curraddr{Universidad Sergio Arboleda, Escuela de Matem\'aticas, Calle 74 No. 14-14, Bogot\'a, Colombia}
\email{hermes.martinez@ima.usergioarboleda.edu.co}


\date{June 23, 2011 and, in revised form November 28, 2012}

\dedicatory{The second author wants to dedicate this paper to Heset Mar\'ia.}




\maketitle
\begin{abstract}
The main result of this paper is the Proposition (\ref{mainresult}). This result establish an explicit ring isomorphism between the twisted Orbifold K-theory ${^{\omega}}K_{orb}([*/G])$ and $R(\D)$ for any element $\omega\in Z^3(G;S^1)$. We also study the relation between the twisted Orbifold K-theories ${^{\alpha}}K_{orb}(\mathcal{X})$ and ${^{\alpha'}}K_{orb}(\mathcal{Y})$ of the Orbifolds $\mathcal{X}=[*/G]$ and $\mathcal{Y}=[*/G']$, where $G$ and $G'$ are different finite groups, and $\alpha\in Z^3(G;S^1)$ and $\alpha'\in Z^3(G';S^1)$ are different twistings. We prove that if $G'$ is an extra-special group of index a prime number $p$ and order $p^n$ (for some $n\in\N$ fixed), under suitable hypothesis over the twisting $\alpha'$, we can obtain a twisting $\alpha$ on the group $(\Z_p)^n$ such that there exists an isomorphism between the twisted K-theories  ${^{\alpha'}}K_{orb}([*/G'])$ and  ${^{\alpha}}K_{orb}([*/(\Z_p)^n])$. 
\end{abstract}

\section{Introduction}
The twisted K-theory is a successful example for the increasing flow of physical ideas into mathematics. Brought from the physical setup, 
the twisted Orbifold K-theory has been, for the last twenty five years, a fruitful field of ideas and development in K-theory and
 algebraic topology. It emerged from two different sources. The first one, the consideration of the D-brane charge on a smooth 
manifold by Witten in \cite{wit} and the second one, the concept of discrete torsion on an Orbifold by Vafa in \cite{vafa}. 
Although, for any element $\alpha\in H^3(\mathcal{X},\Z)$ one can associate the twisted K-theory ${^{\alpha}}K(\mathcal{X})$, its structure 
is simpler if the element $\alpha$ lies in the image of the pullback associated to the map $X\rightarrow *$. In such a case, 
we call this element a \emph{discrete torsion} since we can see it as an element in the cohomology $H^3(G;\Z)$.\\
On the other hand, an Orbifold is a type of generalization of a smooth manifold. It is a topological space locally modeled as a 
quotient of a manifold by an action of a finite group. When $X$ represents an Orbifold, the twisted K-theory is more interesting because 
it is naturally related with equivariant theories if we specialize in the Orbifold $\mathcal{X}=[X/G]$ where $X$ is a smooth manifold and 
$G$ is a compact Lie group acting almost freely on $X$. In the case of Orbifolds, we have another important advantage to work with and it 
is the cohomological counterpart given by the Chen-Ruan cohomology on Orbifolds $H_{CR}^*(\mathcal{X};\C)$ related to K-theory by the Chern
 character. The Chen-Ruan cohomology of Orbifolds has an interesting non trivial internal product which makes it an algebra. This product 
can be presented in the setup of the K-theory to obtain a stringy product on the K-theory of the Orbifold $K_{orb}(\mathcal{X})$ 
(see \cite{bu},\cite{AR}). If the Orbifold considered has the form $[X/G]$, then the twisted Orbifold K-theory can be related to the 
equivariant K-theory of the spaces of fixed points by the $G$-action on $X$.\\
On the other hand, the tensor product defines a product 
\[{^{\alpha}}K_{orb}(\mathcal{X})\otimes{^{\beta}}K_{orb}(\mathcal{X})\rightarrow {^{\alpha +\beta}}K_{orb}(\mathcal{X})\]
for any pair of elements $\alpha$ and $\beta$ in $H^3(\mathcal{X},\Z)$. In fact, one can obtain a stringy product for the twisted K-theory
 on Orbifolds by using the stringy product defined on each space of fixed points, to define an explicit stringy product in each 
${^{\alpha}}K(\mathcal{X})$ for any $\alpha \in H^3(\mathcal{X},\Z)$. Nevertheless, the crucial information to define the stringy product 
on the twisted K-theory of Orbifolds does not lie in $H^3(\mathcal{X},\Z)$; instead it lies in $H^4(\mathcal{X},\Z)$. Given an element 
$\phi$ in $H^4(\mathcal{X},\Z)$, it defines an element $\theta(\phi)$ in $H^3(\wedge\textsl{ X},\Z)$, where $\wedge \mathcal{X}$ is the 
inertia Orbifold associated to $\mathcal{X}$. Hence, we can define a stringy product over the twisted K-theory Orbifold 
${^{\theta(\phi)}}K_{orb}(\wedge\mathcal{X})$ by using a suitable structure of the inertia Orbifold $\wedge\textsl{X}$. One such product structure is based on the map $\theta(\phi)$ called the \emph{inverse transgression map}, which is considered as the inverse of 
the classical transgression map.\\
For this paper, the stringy product in $K_{orb}(\mathcal{X})$ has a trivial expression as we will restrict our observations to the case 
in which $\mathcal{X}=[\{*\}/G]$, where $G$ is a finite group.\\ The main result in this paper is to present an explicit relation between 
the twisted Drinfeld Double $\D$ and the twisted Orbifold K-theory ${^{\omega}}K_{orb}([*/G])$ for an element $\omega$ of discrete torsion (see Section 4 below.). This allows us to relate the twisted Orbifold K-theories ${^{\omega}}K_{orb}([*/G])$ and ${^{\omega'}}K_{orb}([*/G'])$ for, respectively, two different Orbifolds $[*/G]$ and $[*/G']$ and twistings $\omega\in H^4(G;\Z)$, $\omega'\in H^4(G';\Z)$. In order to obtain such a relation, we modify the stringy product defined in \cite{ARZ} by an element in $R_{\alpha_g}(C(g)\cap C(h))$.\\
We want to express our gratitude to Professor Bernardo Uribe for his important suggestions and ideas for our work.
\section{Pushforward map in the twisted representation ring}
In this section we introduce the pushforward map. Although this map can be defined for almost complex manifolds, we will focus only 
on the case of homogeneous spaces $G/H$. To define this map let us recall the Thom isomorphism theorem in equivariant K-theory.
\begin{fact}[\cite{segal1968}, Proposition 3.2]Let $X$ be a compact $G$-manifold and $p:E\rightarrow X$ a complex $G$-vector bundle over 
$X$. There exists an isomorphism \[\phi:K_G^*(X)\rightarrow K_G^*(E,E\setminus E_0)\] \[\phi([F]):=p^*(F)\otimes \lambda_{-1}(E),\]where $E_0$ is the zero section and the class $\lambda_{-1}(E)$ is the Thom class associated to $[E]$. 
\end{fact}

\begin{nota}We need recall how to define the \textit{normal bundle}. If $M$, $N$ are $G$-manifolds (that means a manifold with a smooth $G$-action) and $f:M\rightarrow N$ is a $G$-embedding  we can define a (real) vector bundle $\tau$ such that $df(TM)\oplus\tau\cong TN$ (for details in this construction the reader can consult \cite{td87}). If the map $f$ is not a $G$-embedding we can consider $f:M\rightarrow N\times D^j$ ($D^j$ is the unitary disk in $\mathbb{R}^j$ with the trivial $G$-action) for  sufficiently large $j$ and by corollary 1.10 in \cite{wasserman1969} we can approximate $f$ by an immersion $g_f$, then we define the normal bundle for $f$ as the normal bundle of $g_f$. 
\end{nota}

Now, we proceed to define the pushforward map $f_*:K_G^*(X)\rightarrow K_G^*(Y)$ for a differentiable map $f:X\rightarrow Y$ between 
almost complex $G$-manifolds as follows: let $\tau$ be the normal bundle associated to the application $f:X\rightarrow Y$. We define the 
pushforward, which will be denoted by $f_{\ast}$, as the following composition 
\[K_G^*(X)\xrightarrow{\phi}K_G^*(\tau,\tau\setminus \tau_0)\xrightarrow{j}K_G^*(Y\times D^j,(Y\times D^j)\setminus g_f(X))\xrightarrow{i_\sharp}K_G^*(Y\times D^j)\cong K_G^*(Y),\]
where $\phi$ is the Thom isomorphism. The map $j$ is given by excision , the map $i_\sharp$ is the pullback map induced by the inclusion and the last isomorphism is induced by the natural inclusion. The pushforward map can be defined
 also in the twisted case (see \cite{carey2006}).
Consider the following diagram of inclusions:
\[\begin{diagram}
\node{G}\node{K}\arrow{w}\\\node{H}\arrow{n}\node{H\cap K}\arrow{w}\arrow{n}
\end{diagram}\]
from which we get a diagram of surjections:
\[\begin{diagram}
\node{G/G}\node{G/K}\arrow{w,t}{i_1}\\\node{G/H}\arrow{n,r}{j_1}\node{G/(H\cap K)}\arrow{w,t}{i_2}\arrow{n,r}{j_2}
\end{diagram}\]
Using this diagram we obtain the map: 
\[j_{2*}\circ i_2^*:K_G^*(G/H)\rightarrow K_G^*(G/K)\] 
\[j_{2*}\circ i_2^*([E])=[\lambda_{-1}(\tau_{j_2})\otimes i_2^*(p^*(E))],\]where $\tau_{j_2}$ is the normal bundle of $j_2$, and the map
\[i_1^*\circ j_{1*} :K_G^*(G/H)\rightarrow K_G^*(G/K)\]
\[i_1^*\circ j_{1*} ([E])=[i_1^*(\lambda_{-1}(\tau_{j_1})\otimes i_1^*(p^*(E))].\]
Afterwards, we compare the two applications and we conclude that the \textit{obstruction bundle} is $\lambda_{-1}(i_1^*(\tau_{j_1})/\tau_{j_2}))$. This means that
\begin{equation}\label{hazexceso}
i_1^*\circ j_{1*} ([E])=j_{2*}\circ i_2^*([E])\otimes \lambda_{-1}(i_1^*(\tau_{j_1})/\tau_{j_2})).
\end{equation}
We consider the particular case of the groups $H=C_G(x)$ and $K=C_G(y)$, where $x$ and $y$ are elements in the group $G$, $C_G(x)$ and $C_G(y)$ denote their centralizers in $G$, respectively. Then, by equation (\ref{hazexceso}) we get an obstruction bundle which is denoted as $\gamma_{x,y}$.
\section{Twisted Orbifold K-theory for the Orbifold $[*/G]$}

The goal of this section is to consider a K-theory structure on an Orbifold structure defined by the trivial action of a finite group over the space $\{*\}$. This is a particular case of a more general kind of spaces that are obtained by almost free actions of a compact Lie group over compact manifolds. These spaces naturally have an Orbifold structure that sets a basis for all developments in this paper. When the manifold is one point and the group is finite, all the hypothesis in the already defined theory hold. \\
Let us consider the inertia Orbifold $\wedge[*/G]$ for a finite group $G$. We define the Orbifold K-theory for the Orbifold $[*/G]$ as 
the module: \[K_{orb}([*/G]):=K(\wedge[*/G])\cong \bigoplus_{(g)}K([*/C_G(g)])\cong \bigoplus_{(g)} K_{C_G(g)}(*),\]where $(g)$ denotes the class of conjugation of the 
element $g\in G$ and $C_G(g)$ denotes the centralizer of the element $g\in G$. In this case the Orbifold K-theory introduced in 
\cite{ARZ} turns out to be simply $K_G(*)$, which is additively isomorphic to the group $\bigoplus_{(g)}R(C_G(g))$ (see \cite{AR}), where $R(C_G(g))$ denotes the Grothendieck ring associated to the semi-group of isomorphism classes of linear representations of the group $C_G(g)$, and the sum is taken over conjugacy classes. The product structure in $K_G(*)$ is defined as follows: consider the maps \[e_1:C_G(g)\cap C_G(h)\times C_G(g)\cap C_G(h)\rightarrow C_G(g), e_1(a,b)=a,\] \[e_2:C_G(g)\cap C_G(h)\times C_G(g)\cap C_G(h)\rightarrow C_G(h), e_2(a,b)=b,\] \[e_{12}:C_G(g)\cap C_G(h)\times C_G(g)\cap C_G(h)\rightarrow C_G(gh), e_{12}(a,b)=ab.\]Note that for any element $\tau \in G$, the map $\phi_{\tau}:G\rightarrow G$ defined by $\phi_{\tau}(g)=\tau g\tau^{-1},$  implies that $\phi_{\tau}\circ e_i=e_i\circ(\phi_{\tau},\phi_\tau)$. Thus, the maps $e_i$ are $\phi_{\tau}$-equivariant for any element $\tau \in C_G(g)$. Given $E$ in $R(C_G(g))$ and $F$ in $ R(C_G(h))$, we define the product
\begin{equation}\label{prodstar}
E\star F:=e_{12*}(e_1^*(E)\otimes e_2^*(F)\otimes\gamma_{g,h}) \in R(C_G(gh)).
\end{equation}
Since the action is trivial, it follows from Theorem 2.2 in \cite{segal1968} that $R(C_G(g))=K_{C_G(g)}(*)$. Thus, the product can be seen as \[\star:K_{C_G(g)}(*)\times K_{C_G(h)}(*)\rightarrow K_{C_G(gh)}(*)\] in the setup of equivariant K-theory. We note that the product defined in equation (\ref{prodstar}) is analogous to the stringy and twisted stringy product defined by B. Uribe and the second author in \cite{bu}, in the case in which $G$ is an abelian group. Let $\alpha$ be an cocycle in $Z^3(G;S^1)$, i.e. $\alpha$ is a function $\alpha:G\times G\times G\rightarrow S^1$ such that $\alpha(a,b,c)\alpha(a,bc,d)\alpha(d,c,d)=\alpha(ab,c,d)\alpha(a,b,cd)$ for all $a,b,c,d\in G$. We proceed to define the twisted Orbifold K-theory ${^{\alpha}}K_{orb}([*/G])$.\\
For the global quotient $[X/G]$ and the element $\alpha\in Z^3(G;S^1)$, the twisted Orbifold K-theory is defined as the sum
\begin{equation}\label{TOK}
{^{\alpha}}K_{orb}([X/G]):=\bigoplus_{g\in C}{^{\alpha_g}}K_{C_G(g)}(X^g)
\end{equation}
where $C$ is a set of representatives of the conjugacy classes in $G$ and $\alpha_g$ is the inverse transgression map (see below for details). In particular, if $G$ is an abelian group, the set $C$ is the group $G$. For every group $H$ and $\beta\in Z^2(H;S^1)$ we take its associated group $H_{\beta}$ given by the central extension
\begin{equation}
1\rightarrow S^1\rightarrow H_{\beta} \rightarrow H\rightarrow 1.
\end{equation}Recall that the group $H_{\beta}$ is the set $S^1\times H$, with the group operation defined by \[(s_1,h_1)*(s_2,h_2):=(s_1s_2\beta(h_1,h_2),h_1h_2).\]The twisted equivariant K-theory ${^{\beta}}K_{H}(X)$ is defined as the class of $H_{\beta}$-equivariant vector bundles such that the action of the center $S^1$ restricts to multiplication on the fibres. In the case of the space $X=\{*\}$, the twisted equivariant K-theory ${^{\beta}}K_{H}(*)$ coincides with $R_{\beta}(H)$, the Grothendieck ring of classes of projective representations for the group $H$ (see \cite{Karpi} for a precise definition of $R_{\beta}(H)$).\\
Returning to the case of the Orbifold $[*/G]$ for a finite group $G$, the twisted Orbifold K-theory defined in equation \ref{TOK} takes the form 
\begin{equation}
{^{\alpha}}K_{orb}([*/G]):=\bigoplus_{g\in C}{^{\alpha_g}}K_{C_G(g)}(*)\cong \bigoplus_{g\in C}R_{\alpha_g}(C_G(g)).
\end{equation}
\subsection{Inverse transgression map}
We review the inverse transgression map for finite groups to describe the multiplicative structure in the module ${^{\alpha}}K_{orb}([*/G])$. Throughout this section we follow the development presented in section 3.2. in \cite{bu}. Let us recall the definition of the inverse transgression map for a global quotient $[M/G]$: For $g\in G$, consider the action of $C_G(g)\times \Z$ on $M^g=\{x\in M|gx=x\}$ given by $(h,m)\cdot x:=hg^mx$ and the homomorphism\\
\begin{center}\begin{tabular}{cccc}$\psi_g$&$:C_G(g)\times \Z$&$\rightarrow$&$G$\\
&$(h,m)$&$\mapsto$&$hg^m$\\ 
\end{tabular}\\ \end{center}
Thus, the inclusion $i_g:M^g\rightarrow M$ becomes a $\psi_g$ equivariant map and induces a homomorphism\[i_g^*:H_G^*(M;\Z)\rightarrow H^*_{C_G(g)\times\Z}(M^g;\Z).\] From the isomorphisms \begin{align} H^*_{C_G(g)\times\Z}(M^g;\Z)&\cong H^*(M^g\times_{C_G(g)}\times EC_G(g)\times B\Z;\Z)\\ &\cong H^*_{C_G(g)}(M^g;\Z)\otimes_{\Z}H^*(S^1;\Z)\end{align} we have, for each $k$, \[i_g^*:H^k_G(M;\Z)\rightarrow H^k_{C_G(g)}(M^g;\Z)\oplus H^{k-1}_{C_G(g)}(M^g;\Z).\]Hence, we define the inverse transgression map as the application induced by projecting on the second factor\[\tau_g:H_G^k(M;\Z)\rightarrow H^{k-1}_{C_G(g)}(M^g;\Z).\]In the particular case that $[M/G]=[*/G]$ the definition above turns into:
\begin{defi}\label{defitran}
For any element $\alpha\in Z^3(G;S^1)$, the inverse transgression map is defined as the application
\[\tau_g:H_G^k(*;\Z)\rightarrow H_{C_G(g)}^{k-1}(*;\Z)\]induced by $\tau_g$ on each $k$.
\end{defi}
\subsection{Product in the twisted case}
Take $\alpha\in Z^3(G;\Z)$. Let us consider the Orbifold $[*/G]$ where $G$ is a finite group. Now, consider the module:
\begin{equation}\label{pretorcida}
{^{\alpha}}K_{orb}([*/G]):=\bigoplus_{g\in C}R_{\alpha_g}(C_G(g)),
\end{equation}
where $\alpha_g\in H^2(C(g);\Z)$ denotes the inverse transgression map of $\alpha$. The goal of this section is to define an associative product for this module, i.e. we show that it's possible to endow the module ${^{\alpha}}K_{orb}([*/G])$ with a ring structure. For simplicity, we denote $C_G(g)$ as $C(g)$. Consider the inclusion maps of groups \[i_g:C(g)\cap C(h)\rightarrow C(g), i_h:C(g)\cap C(h)\rightarrow C(h)\]and\[i_{gh}:C(g)\cap C(h)\rightarrow C(gh)\]for $g,h\in G$. These maps induce the restriction maps: 
\[i_g^*:H^2(C(g);S^1)\rightarrow H^2(C(g)\cap C(h);S^1),\]
\[i_h^*:H^2(C(h);S^1)\rightarrow H^2(C(g)\cap C(h);S^1).\]
and the morphism $i_{gh}$ induces a map $(i_{gh})_*:H^2(C(g)\cap C(h);S^1)\rightarrow H^2(C(gh);S^1)$, which is the induction morphism in group cohomology (See for example \cite{td87}).\\
Given $E\in R_{\alpha_g}(C(g))$, we consider it as a $C(g)_{\alpha_{g}}$-module that restricts to multiplication on the fibres over $S^1$. Therefore, we get the following commutative diagram for the inclusion $i_g$ and the identity map $s$ on $S^1$:
\begin{equation}
\begin{tabular}{c c c c c c c c c}
  $1$&$\rightarrow$& $S^1$& $\rightarrow$ &$C(g)_{\alpha_{g}}$ & $\rightarrow$ & $C(g)$ &$\rightarrow$& 1 \\
  &&$\uparrow s$&&$\uparrow (s, i_g)$&&$\uparrow i_g$&&\\
  $1$&$\rightarrow$& $S^1$& $\rightarrow$ &$(C(g)\cap C(h))_{i_g^*(\alpha_{g})}$ & $\rightarrow$ & $C(g)\cap C(h)$ &$\rightarrow$& 1\\
\end{tabular}\end{equation}
This implies that any $C(g)_{\alpha_{g}}$-module restricts to a $(C(g)\cap C(h))_{i_g^*(\alpha_{g})}$-module, which is denoted by $i_g^*(E)$. In particular, for any $(E,F)\in R_{\alpha_g}(C(g))\times R_{\alpha_h}(C(h))$, we get the following map:\\
\begin{align*}
R_{\alpha_g}(C(g))\times R_{\alpha_h}(C(h))&\rightarrow R_{i_g^*(\alpha_g)}(C(g)\cap C(h))\times R_{i_h^*(\alpha_h)}(C(g)\cap C(h))\\
(E,F)&\mapsto  (i_g^*(E),i_h^*(F))\\
\end{align*}
Now, from the central extensions: 
\[1\rightarrow S^1 \rightarrow (C(g)\cap C(h))_{i_g^*(\alpha_{g})} \rightarrow C(g)\cap C(h) \rightarrow 1,\]
\[1\rightarrow S^1 \rightarrow (C(g)\cap C(h))_{i_h^*(\alpha_{h})} \rightarrow C(g)\cap C(h) \rightarrow 1\]
induced by $i_g^*(\alpha)$ and $i_h^*(\alpha)\in H^2(C(g)\cap C(h);S^1))$ we get that:\[{\small 1\rightarrow S^1\times S^1 \rightarrow (C(g)\cap C(h))_{i_g^*(\alpha_{g})}\times (C(g)\cap C(h))_{i_h^*(\alpha_{h})} \rightarrow C(g)\cap C(h)\times C(g)\cap C(h) \rightarrow 1.}\]For $E\in R_{i_g^*(\alpha_g)}(C(g))$ and $F\in R_{i_g^*(\alpha_g)}(C(g))$, the tensor product $E\otimes F$ is naturally a $(C(g)\cap C(h))_{i_g^*(\alpha_{g})}\times (C(g)\cap C(h))_{i_h^*(\alpha_{h})}$-module that restricts to multiplication on the fibres by elements of $S^1$. By considering the action restricted to the diagonal
\[\Delta(C(g)\cap C(h))\subset C(g)\cap C(h)\times C(g)\cap C(h),\]we get the central extension:
\[1\rightarrow S^1 \rightarrow (C(g)\cap C(h))_{i_g^*(\alpha_{g})i_h^*(\alpha_{h})} \rightarrow \Delta(C(g)\cap C(h)) \rightarrow 1\] which corresponds to the element $i_g^*(\alpha_{g})i_h^*(\alpha_{h})\in H^2(C(g)\cap C(h);S^1)$. Thus, the following holds:
\begin{align*}
R_{i_g^*(\alpha_g)}(C(g)\cap C(h))\times R_{i_h^*(\alpha_h)}(C(g)\cap C(h))&\rightarrow R_{i_g^*(\alpha_g)i_h^*(\alpha_h)}(C(g)\cap C(h))\\
(E,F)&\mapsto E\otimes F.\\
\end{align*}
Now, since $i_g^*(\alpha_g)i_h^*(\alpha_h)=i_{gh}^*(\alpha_{gh})$ in $H^2(C(g)\cap C(h);S^1)$ are cohomologous cocycles, (see \cite{ARZ} Proposition 4.3), it follows that:
\[R_{i_g^*(\alpha_g)i_h^*(\alpha_h)}(C(g)\cap C(h))\cong R_{i_{gh}^*(\alpha_{gh})}(C(g)\cap C(h)).\] Therefore, the induction map can be defined as: 
\begin{align*}
R_{i_{gh}^*(\alpha_{gh})}(C(g)\cap C(h))&\rightarrow R_{(i_{gh})_*i_{gh}^*(\alpha_{gh})}(C(gh))\\
A&\mapsto Ind_{C(g)\cap C(h)}^{C(gh)}(A).\\
\end{align*}
Thus, a product on the module (\ref{pretorcida}) can be obtained from the previously described morphisms to get:
\begin{equation}\label{productotorcido}
R_{\alpha_g}(C(g))\times R_{\alpha_h}(C(h))\rightarrow R_{\alpha_{gh}}(C(gh))
\end{equation}
defined by:\[(E,F)\mapsto E\star_{\alpha} F:=Ind_{C(g)\cap C(h)}^{C(gh)}(i_g^*(E)\otimes i_h^*(F)\otimes \gamma_{g,h}),\] where $\gamma_{g,h}$ is defined as the excess bundle as in Section 2.
\begin{defi}
By using the restriction notation, we define the twisted stringy product in the module ${^{\alpha}}K_{orb}([*/G])$ as the map:
\begin{equation}
\begin{aligned}R_{\alpha_g}\hspace{-0.05cm}(C(g))\!\times\! R_{\alpha_h}\hspace{-0.05cm}(C(h))&\rightarrow R_{\alpha_g\alpha_h}\hspace{-0.05cm}(C(gh))\\
(E,F)\mapsto & { I^{C(\!gh\!)}_{C(\!g\!)\cap C(\!h\!)}(Res^{C(\!g\!)}_{C(\!g\!)\cap C(\!h\!)}(E)\!\otimes\! Res^{C(\!h\!)}_{C(\!g\!)\cap C(\!h\!)}(F)\otimes \gamma_{g,h})}
\end{aligned}
\end{equation}
where $Res_{C(g)\cap C(h)}^{C(g)}$ denotes the restriction of $\alpha$-twisted representations of $C(g)$ to $i_g(\alpha)$-representations of $C(g)\cap C(h)$ (respectively for $h$).

\end{defi}
\section{Twisted Orbifold K-theory and the algebra $D^{\omega}(G)$}
The goal of this section is to give an introduction of the \emph{twisted Drinfeld double} $D^{\omega}(G)$ and to show an explicit 
relation with the twisted Orbifold K-theory. Our main reference is \cite{sw1}. Let us recall the definition and the main properties 
of the twisted Drinfeld double to clarify the nature of this structure and its representations. From a different point of view, we can also obtain some properties of the stringy product defined on the sections above, using the properties of the representations of the twisted Drinfeld double. Namely, the fact that the Grothendieck ring of these representations is isomorphic to the twisted Orbifold K-theory with the structure induced by the stringy product, which will be proven in section 4.2. In particular, because of the associativity of the tensor product
of the $\D$-modules we can obtain a proof of the associativity of the stringy product defined above (see \ref{asociativo} below). \\
Let $G$ be a finite group and $k$ an algebraically closed 
field. Let $\omega$ be an element in $Z^{3}(G,k^*)$, that is, a function $\omega:G\times G\times G\rightarrow k^*$ such that 
\[\omega(a,b,c)\omega(a,bc,d)\omega(d,c,d)=\omega(ab,c,d)\omega(a,b,cd)\text{   for all }a,b,c,d\in G.\] We define the 
quasi-triangular quasi-Hopf algebra $D^{\omega}(G)$ as the vector space $(kG)^*\otimes (kG)$, where $(kG)^*$ denotes the dual of 
the algebra $kG$ (see \cite{drin}) and the algebra structure in $D^{\omega}(G)$ is given as follows: consider the canonical basis 
$\{\delta_g\otimes \bar{x}\}_{g,x\in G}$ of $D^{\omega}(G)$, where $\delta_g$ is the function such that $\delta_g(h)=1$ if $h=g$ 
and $0$ otherwise. We denote $\delta_g\otimes \bar{x}=:\delta_g\bar{x}$. Now, we define the product of elements in the basis by:
\begin{equation}
(\delta_g\bar{x})(\delta_h\bar{y})=\omega_g(x,y)\delta_g\delta_{xhx^{-1}}\overline{xy}
\end{equation}
where $\omega_g$ is the image of $\omega$ via the inverse transgression map of the element $g\in G$ as in definition (\ref{defitran}). The multiplicative identity for this product is the element $1_{\D}=\bigoplus_{g\in G}\delta_g\bar{1}$. Now, we use the notation $\delta_g$ for the element $\delta_g\bar{1}$. The coproduct $\Delta:\D\rightarrow\D\otimes\D$ in the algebra $D^{\omega}(G)$ is defined by the application
\begin{equation}
\Delta(\delta_g\bar{x})=\bigoplus_{h\in G}\gamma_x(h,h^{-1}g)(\delta_h\bar{x})\otimes(\delta_{h^{-1}g}\bar{x}),
\end{equation}
where \[\gamma_x(h,l)=\frac{\omega(h,l,x)\omega(x,x^{-1}hx,x^{-1}lx)}{\omega(h,x,x^{-1}lx)}.\]
The algebra $D^{\omega}(G)$ endowed with these operations is usually called the twisted Drinfeld Double.
\subsection{Representations of $D^{\omega}(G)$}
Let $U, V$ be modules over the algebra $D^{\omega}(G)$. Consider the tensor product $U\otimes V$ as a $D^{\omega}(G)$-module endowed with the action from $D^{\omega}(G)$, induced by the coproduct $\Delta$. Note that the field $k$ can be considered as a trivial $D^{\omega}(G)$-module, which is the multiplicative identity for the tensor product of $\D$-modules. In particular, for $k=\C$, we define the ring of representations $R(\D)$ of $\D$ as the $\C$-algebra generated by the set of isomorphism classes of $\D$-modules with the direct sum of modules as the sum operation and the tensor as the product operation. We define the ideal $R_0(\D)$ generated by all combinations $[U]-[U']-[U'']$ ([.] denotes the isomorphism class) where $0\rightarrow U'\rightarrow U\rightarrow U''\rightarrow 0$ is a short exact sequence of $\D$-modules. Now, we define the Grothendieck ring $R(\D)$ as the quotient between $Rep(\D)$ and the ideal $R_0(\D)$.\\
The algebra $\D$ is quasi-triangular with:
\[A=\bigoplus_{g,h\in G}\delta_g\bar{1}\otimes\delta_h\bar{g}\text{, \, and         }A^{-1}=\bigoplus_{g,h\in G}\omega_{ghg^{-1}}(g,g^{-1})^{-1}\delta_g\bar{1}\otimes\delta_h\overline{g^{-1}}.\] Thus, $A\Delta(a)A^{-1}=\sigma(\Delta(a))$ for all $a\in\D$, where $\sigma$ is the automorphism that exchanges the images in the coproduct. Therefore, if $U$ and $V$ are $\D$-modules, this equation implies that $U\otimes V$ and $V\otimes U$ are isomorphic as $\D$-modules, that is, the algebra $R(\D)$ is commutative. Now, assume that $\beta:G\times G\rightarrow \C^*$ is a cochain with coboundary:\[\delta\beta(a,b,c)=\beta(b,c)\beta(a,bc)\beta(ab,c)^{-1}\beta(a,b)^{-1}.\] Then, the algebra $D^{\omega\delta\beta}(G)$ is isomorphic to $\D$ given through the map: \[\upsilon(\delta_g\bar{x})=\frac{\beta(g,x)}{\beta(x,xgx^{-1})}\delta_g\bar{x}.\]In particular, we get the isomorphism:\[\upsilon^*:R(D^{\omega\delta\beta}(G))\stackrel{\cong}{\rightarrow}R(\D).\]Next, we consider the following Theorem (cf. \cite{will}, Thm. 19):
\begin{teor}\label{teoremaaditivo}
The ring $R(\D)$ is additively isomorphic to the ring \[\bigoplus_{(g)\subset G}R_{\omega_g}(C(g)),\]where $(g)$ denotes the conjugacy class of $g\in G$
\end{teor}
\begin{proof}
For all $x\in G$, we take the subspaces $S^{\omega}(x):=\bigoplus_{g\in C(x)}\C\delta_x\bar{g}$ and $D^{\omega}(x):=\bigoplus_{g\in G}\C\delta_x\bar{g}$ of $\D$. It holds that $S^{\omega}(x)$ is a subalgebra of $\D$ with identity element $\delta_x\bar{1}$ such that, from the product defined in $\D$, it follows that $S^{\omega}(x)\cong R_{\omega_x}C(x)$ where $R_{\omega_x}C(x)$ is defined in \cite{Karpi}. Given $(g)\subset G$, consider $D^{\omega}((g)):=\bigoplus_{h\in (g)}D^{\omega}(h)$. Note that $\D\cong \bigoplus_{(g)\subset G}D^{\omega}((g))$ (additively). On the other hand, for an element $h$ in a fixed conjugacy class $(g)$, take a $S^{\omega}(h)$-module $U$, i.e. a $R_{\omega_h}C(h)$-module and define the map:
\[U\mapsto U\otimes_{S^{\omega}(h)}D^{\omega}(h),\]whose image is a $D^{\omega}((g))$-module if we take the action of $D^{\omega}((g))$ over $U\otimes_{S^{\omega}(h)}D^{\omega}(h)$ as right multiplication in the second factor. However, for $V$ a $D^{\omega}((g))$-module, we define the map:
\[V\mapsto V\delta_h\bar{1}.\]
Note that the image for this map is a $R_{\omega_h}C(h)$-module. Thus, there is an equivalence between $R_{\omega_h}C(h)$-modules and $D^{\omega}((g))$-modules. Therefore, from Theorem (\cite[Thm. I.3.2]{Karpi}), it follows that for any $h\in (g)$:
\[R_{\omega_h}(C(h))\cong R(D^{\omega}((g))),\]and the theorem follows.
\end{proof}
From \cite{dpr} we get that it is possible to explicitely describe the morphism using the \emph{induction DPR} which is defined on each $R_{\alpha}(C(g))$ for $g\in G$. Namely, let $(\rho, V)$ be a twisted representation of the group $C(g)$ and define the representation $\psi((\rho,V)):=(\pi_{\rho},A)$ of $\D$ as given by the formula:
\begin{equation}\label{Indpr}
A:=Ind_{C(g)}^G(V) \text{,      and       }\pi_{\rho}:=\pi_{\rho}(\delta_k\bar{x}) x_j\otimes v =\delta_{k}\delta_{x_s g x_s^{-1}}
\frac{\omega_k(x,x_j)}{\omega_k(x_s,r)} x_s\otimes \rho(r)v.     
\end{equation} where $x_j$ is a representative of a class in $G/C(g)$, $r\in C(g)$ and the element $x_s$ is a representative of a class in $G/C(g)$, such that $xx_j=x_sr$.
\subsection{Relation between $R(\D)$ and the twisted K-theory of the Orbifold $[*/G]$}
Let us consider an element $\omega\in Z^3(G;S^1)$. By the equiation \ref{pretorcida} the twisted
 Orbifold K-theory of the Orbifold $[*/G]$ is the ring:
\begin{equation}
{^{\omega}}K_{orb}([*/G])=\bigoplus_{(g)\subset G}{^{\omega_g}}K_{C(g)}(*)\cong \bigoplus_{(g)\subset G}R_{\omega_g}(C(g))
\end{equation}
By theorem (\ref{teoremaaditivo}), there exists an additive isomorphism between $R(\D)$ and the twisted Orbifold K-theory 
${^{\omega}}K_{orb}([*/G])$. We will show that if we endow this ring with the twisted product $\star_{\alpha}$, then the additive 
isomorphism is in fact a ring isomorphism. 
The DPR induction is defined as $(I^G_{C(g)}(E),\rho_{\pi})$ where $(E,\pi)$ is an element in $R_{\omega_g}(C(g))$. Let us consider 
two elements $E$ and $F$ in $R_{\omega_g}(C(g))$ and $R_{\omega_h}(C(h))$ respectively. The tensor product of the DPR-induction of these elements can be related to the twisted product $\star$ via the Frobenious reciprocity as follows:
\begin{align*} I^G_{C(g)}(E)\otimes I^G_{C(h)}(F) &\cong I^G_{C(g)}(E\otimes R^G_{C(g)}(I^G_{C(h)}(F)))\\
&\cong I^G_{C(g)}(E\otimes I^G_{C(g)}(R^{C(h)}_{C(g)\cap C(h)}(F)\otimes \gamma_{g,h}))\\
&\cong I^G_{C(g)}(I^{C(g)}_{C(g)\cap C(h)}(R^{C(g)}_{C(g)\cap C(h)}(E)\otimes R^{C(h)}_{C(g)\cap C(h)}(F)\otimes \gamma_{g,h}))\\
&\cong I^G_{C(g)\cap C(h)}(R^{C(g)}_{C(g)\cap C(h)}(E)\otimes R^{C(h)}_{C(g)\cap C(h)}(F)\otimes \gamma_{g,h})\\
&\cong I^G_{C(gh)}(I^{C(gh)}_{C(g)\cap C(h)}(R^{C(g)}_{C(g)\cap C(h)}(E)\otimes R^{C(h)}_{C(g)\cap C(h)}(F)\otimes \gamma_{g,h}))\\
&\cong I^{G}_{C(gh)}(E\star_{\omega} F)
\end{align*}
By the above relation, we conclude the following
\begin{prop}\label{mainresult} There exists a ring isomorphism:
\[({^{\omega}}K_{orb}([*/G]),\star_{\omega})\cong(R(\D),\otimes).\] 
\end{prop}
\begin{proof}
We have that the DPR induction defines a morphism $\phi:\bigoplus_{(g)\subset G}R_{\omega}(C(g))\rightarrow R(\D)$. Moreover, we proved above that for $E\in R_{\omega_g}(C(g))$ and $F\in R_{\omega_h}(C(h))$, we have that \[\phi(E)\otimes\phi(F)=\phi(E\star_{\omega}F),\]i.e, it is a ring homomorphism. By Theorem (\ref{teoremaaditivo}), the result follows.
\end{proof}
\begin{coro}\label{asociativo}
The stringy product $\star_{\omega} $ is associative.
\end{coro}

\section{Twisted K-theory for an extra special $p$-group}
The goal of this section is to establish a relation between the twisted Orbifold K-theories for the Orbifolds $[*/H]$ and $[*/G]$, where $H$ is an extra special group with exponent $p$, order $p^{2n+1}$ and $G=(\Z_p)^{2n+1}$. For an odd prime number $p$, a $p$-group $H$ is called \emph{extra special} if its center $Z(H)$ is a cyclic group of order $p$, that is $Z(H)\cong \Z_p$, and  $H/Z(H)$ is an elementary abelian group. Any extra special $p$-group has order $p^{2n+1}$ for some $n\in \mathbb{N}$. On the other hand, for any $n$ there exist two extra special groups of order $p^{2n+1}$ such that a group has exponent $p$ and the other group has exponent $p^2$. The motivation for these kind of relations comes from works like \cite{gmn} where these relations are studied for $p=2$ and to some extent by results due to A. Duman  in \cite{duman}. However, there exists a deeper interest to study these kinds of relations by establishing correspondences with the twisted Drinfeld algebras. In particular, we shall remark 
the work done by D. Naidu and D. Nikshych \cite{nn}, which we consider of utmost importance for obtaining the results of this section. In fact, we use the following result due to them (\cite{nn}, Corollary 4.20):
\begin{teor}\label{teorema equivalencia drinfeld}
Let $H$ be a finite group, $\omega'\in Z^3(H;S^1)$ such that 
\begin{itemize}
\item $H$ contains an abelian normal subgroup $K$,
\item $\omega'|_{K\times K\times K}$ is trivial in cohomology (in $H^3(K;S^1))$,
\item there exists a $H$-invariant $2$-cochain $\mu$ over $H$ such that $\delta(\mu)|_{K\times K\times K}=\omega'|_{K\times K\times K}$.
\end{itemize}
Then, there exists a group $G$ and an element $\omega\in Z^2(G;S^1)$ such that $R(\D)\cong R(D^{\omega'}(H))$.
\end{teor}
From the established relation in the previous Chapter between the twisted Drinfeld's algebras and the twisted Orbifold K-theory, we get the following corollary under the same assumptions as in the last theorem.
\begin{coro}
There exists a ring isomorphism:\[{^{\omega}}K_{orb}([*/G])\cong {^{\omega'}}K_{orb}([*/H]).\]
\end{coro}
Now, we follow with a nice application of this result.
\begin{prop}
Let $H$ be an extra special group with order $p^{2n+1}$ and exponent $p$. Then
\[K_{orb}([*/H])\cong {^{\omega}}K_{orb}([*/(\Z_p)^{2n+1}])\] for some non trivial twisting $\omega$.
\end{prop}
\begin{proof}
Let $H$ be an extra special group. From definition we may assume $K=Z(H)\cong \Z_p$. Now, suppose there exists $\mu\in C^2(H;S^1)$ such that $\delta(\mu)|_{K\times K\times K}=\omega'|_{K\times K\times K}$, which is $H$-invariant, that is, if we take the action of $H$ on the $2$-cochains $C^2(H;S^1)$ defined by ${^{y}}\mu:=\mu(yx_1y^{-1},yx_2y^{-1})$, then ${^{y}}\mu=\mu$ in $C^2(H;S^1)$ for all $y\in H$. Now, since $K=Z(H)$, then it follows that ${^{y}}\mu|_K=\mu|_K$ for all $y\in H$. Thus, for all $y\in H$ there exists a $1$-chain $\eta_y$ on $H$ such that $\delta\eta_y=\frac{{^{y}\mu}}{\mu}=1$. Since $K$ is abelian, we can define the following  map:
\begin{align*}
\nu:&H/K\times H/K\rightarrow C^1(H;S^1)\\
&(y_1,y_1)\mapsto\frac{{^{y_2}}\eta_{y_1}\eta_{y_2}}{\eta_{y_1y_2}}.
\end{align*}
\begin{lema}[\cite{n}, Lemma 4.2, Corollary 4.3]
The function $\nu$ defines an element in $H^2(H/K;\hat{K}))$. 
\end{lema}

However, this element represents a short exact sequence:\[1\rightarrow \hat{K}\rightarrow \hat{K}\times_{\nu}H/K\rightarrow H/K\rightarrow 1,\] where the product in $\hat{K}\times_vH/K$ is defined by the formula \[(\rho_1,x_1)(\rho_2,x_2):=(\nu(x_1,x_2)\rho_1\rho_2,x_1x_2).\]
Now, the element $\omega\in Z^3(G;S^1)$ with $G:=\hat{K}\times_{\nu}H/K$ for all $(\rho_1,x_1),(\rho_2,x_2),(\rho_3,x_3)\in\hat{K}\times_{\nu}H/K$ is defined by the formula:
\[\omega((\rho_1,x_1)(\rho_2,x_2)(\rho_3,x_3)):=(\nu(x_1,x_2)(u(x_3)))(1)\rho_1(k_{x_2,x_3}),\] where $u:H/K\rightarrow H$ is a function such that, when it is composed with the projection $p:H\rightarrow H/K$, it follows that $p(u(x))=x$ and $k_{x_2,x_3}\in H$ is an element that satisfies $u(x_1)u(x_2)=k_{x_1,x_2}u(p(u(x_1)u(x_2)))$. \\
Clearly, when $\omega'$ is the trivial $3$-cocycle, we can choose $\mu$ to be trivial and so $\nu$ is also trivial. By definition of an extra special group, $H/K$ is an elementary abelian group and if $\nu$ is trivial, it follows easily that $\hat{K}\times_{\nu}H/K\cong (\Z_p)^{2n+1}$. It remains to show that $\omega$ is non trivial in $H^3(G;S^1)$. Take $H=\{h_1,\dots,h_{p^{2n+1}}\}$, $K=Z(H)=\{z_1,\dots,z_p\}$ and $\hat{K}=\{\rho_1,\dots,\rho_p\}$, with $\rho_i$ non trivial for $i\neq 1$. Denote the quotient group $H/K=\{x_1K,\dots, x_{p^{2n}}K\}$, with $x_1=1_{H/K}$. Now, we define the function $u:H/K\rightarrow H$ such that $u(x_iK)=x_i(z_i^{-1})$ for $x_iK\in H/K$, $z_j\in K$. Consider the element $((\rho,x_iK),(\rho,x_iK),(\rho,x_iK))$ with $\rho\in \hat{K}$ fixed and non trivial. Since $\nu$ is trivial, the element $\omega$ is reduced to $\rho(k_{x_iK,x_iK})=\rho(z_i)\neq 1$, which implies that $\omega$ is non trivial.\end{proof}
 
\subsection{Twisted Orbifold K-theory for the Orbifold $[*/(\Z_p)^n]$}
With the above result, to calculate the Orbifold K-theory structure for $[*/H]$, with $H$ an extra-special $p$-group, we only have to calculate the twisted Orbifold K-theory for $[*/(\Z_p)^n]$ and a twist element in $H^3((\Z_p)^n;S^1)$, following the constructions presented in Section 3. Because all those constructions are based on the inverse transgression map, we proceed to give an explicit way of calculating it. Afterwards, we shall present an example with a particular twist element, that has no trivial inverse transgression map. 
\subsubsection{Inverse transgression map for the group $(\Z_p)^n$}
Let us consider the following commutative diagram given by two natural short exact sequences:
\begin{equation}
\begin{tabular}{cccccccccccc}
$0$&$\rightarrow$&$\Z$&$\stackrel{\times p}{\rightarrow}$&$\Z$&$\stackrel{\pi}{\rightarrow}$&$\Z_p$&$\rightarrow$&$0$\\
&&$\downarrow$&&$\downarrow$&&$\downarrow$&&\\
$0$&$\rightarrow$&$\Z_p$&$\stackrel{\times p}{\rightarrow}$&$\Z_{p^2}$&$\stackrel{\tau}{\rightarrow}$&$\Z_p$&$\rightarrow$&$0$\\
\end{tabular}
\end{equation}
where the applications $\pi$ and $\tau$ are the natural projections. Idem for the downarrow applications. These two exact sequences in the diagram above induce long exact sequences:
\begin{equation}\label{suclar}
\cdots\rightarrow H^{k-1}(BG;\Z_p)\stackrel{\partial}{\rightarrow}H^k(BG;\Z)\stackrel{(\times p)_*}{\rightarrow}H^k(BG;\Z)\stackrel{(\pi)_*}{\rightarrow}H^{k}(BG;\Z_p)\stackrel{\partial}{\rightarrow}H^{k+1}(BG;\Z)\rightarrow\cdots
\end{equation}
and
\begin{equation}\label{suclarb}
\cdots\rightarrow H^{k-1}(BG;\Z_p)\stackrel{\beta}{\rightarrow}H^k(BG;\Z_p)\stackrel{(\times p)_*}{\rightarrow}H^k(BG;\Z_{p^2})\stackrel{(\tau)_*}{\rightarrow}H^{k}(BG;\Z_p)\stackrel{\beta}{\rightarrow}H^{k+1}(BG;\Z_p)\rightarrow\cdots
\end{equation}
 \begin{nota}\label{nota1}The connection morphism $\beta$ of the long exact sequence (\ref{suclarb}) is known as the Bockstein map. Such a morphism induces an application $\beta:H^*(BG;\Z_p)\rightarrow H^*(BG;\Z_p)$ which has the multiplicative property: \[\beta(xy)=\beta(x)y+(-1)^{deg(x)}x\beta(y).\]
\end{nota}
Since $G$ is a $p$-group, $H^k(BG;-)$ is also a $p$-group and this implies that the morphism $(\times p)_*$ in the long exact sequences (\ref{suclar}) and (\ref{suclarb}) is the zero map. Thus, $\pi_*$ and $\tau_*$ are injective maps and $H^k(BG;\Z)\cong H^k(BG;\Z_{p^2})$. On the other hand, by the exactness of the sequence (\ref{suclarb}), $H^k(BG;\Z_{p^2})\cong Ker(\beta:H^k(BG;\Z_p)\rightarrow H^{k+1}(BG;\Z_p))$ and then \[H^k(BG;\Z)\cong Ker(\beta:H^k(BG;\Z_p)\rightarrow H^{k+1}(BG;\Z_p)).\] 
\subsubsection{Relation to the inverse transgression map}
By definition, the inverse transgression map $\tau_g$ is an application defined between the groups $H^k(BG;\Z)$ and $H^{k-1}(BC_G(g);\Z)$. Since $G=(\Z_p)^n$ is an abelian group, we have that the inverse transgression map can be factorized as \small \begin{equation*}\tilde{\tau}_g:Ker(\beta:H^k(BG;\Z_p)\rightarrow H^{k+1}(BG;\Z_p))\rightarrow Ker(\beta:H^{k-1}(BG;\Z_p)\rightarrow H^{k}(BG;\Z_p)).\end{equation*}\normalsize
Consider the cohomology ring $H^*(BG;\Z_p)\cong\F_p[x_1,\dots,x_n]\otimes\Lambda[y_1,\dots,y_n]$ with $|x_i|=2$ and $|y_i|=1$ for $i=1,\dots,n$. By the calculations above we need to find a polynomial $p(x_1,\dots,x_n,y_1,\dots,y_n)\in \F_p[x_1,\dots,x_n]\otimes\Lambda[y_1,\dots,y_n]$ of degree $k$, such that $\beta(p)=0$ and $\tilde{\tau}_g(p)\neq 0$ for some $g\in G$.\\
To obtain the desired polynomial, first we do the calculation of the inverse transgresion map. Take an element $g=(a_1,\dots,a_n)\in G$ and consider the map:
\[G\times\Z\rightarrow G\times \left\langle g\right\rangle\rightarrow G\] defined by \[(h,m)\mapsto (h,g^m)\mapsto hg^m.\]
At the level of cohomology we get:
\begin{equation}
\begin{tabular}{ccccc}
$H^*(BG;\F_p)$&$\rightarrow$&$ H^*(BG\times B(\Z_p);\F_p)$&$\rightarrow$&$ H^*(BG\times B\Z;\F_p)$\\
$x_i$&$\mapsto$&$ x_i+a_iw$&$\mapsto$&$ x_i$\\
$y_i$&$\mapsto$&$ y_i+a_iz$&$\mapsto$&$ y_i+a_iz$
\end{tabular}
\end{equation}
where \[H^*(BG;\F_p)=\F_p[x_1,\dots,x_n]\otimes\Lambda[y_1,\dots,y_n],\] \[H^*(BG\times B(\Z_p);\F_p)=\F_p[x_1,\dots,x_n,w]\otimes\Lambda[y_1,\dots,y_n,z]\]and\[H^*(BG\times B\Z;\F_p)=\F_p[x_1,\dots,x_n]\otimes\Lambda[y_1,\dots,y_n,z].\]
Now, for the products $x_iy_j,x_ix_j, y_iy_j\in H^*(BG;\F_p)$ we can obtain the calculation of the inverse transgression maps. For the first product $x_iy_j$ we get
\footnotesize
\begin{equation}
\begin{aligned}
(x_iy_j)&\mapsto(x_i+a_iw)(y_j+a_iz)=x_iy_j+x_ia_jz+a_iwy_j+a_ia_jwz&\mbox{  in  } H^*(BG\times B(\Z_p);\F_p)\\
&\mapsto x_iy_j+x_ia_jz&\mbox{  in  } H^*(BG\times B\Z;\F_p).
\end{aligned}
\end{equation}
\normalsize
and from definition (\ref{defitran}) it follows that $\tilde{\tau}_g(x_iy_j)=x_ia_j$.
For the second product $x_ix_j$ we get
\begin{equation}\label{invertrans1}
\begin{aligned}
(x_ix_j)&\mapsto(x_i+a_iw)(x_j+a_iw)=x_ix_j+x_ia_jw+a_iwx_j&\mbox{  in  } H^*(BG\times B(\Z_p);\F_p)\\
&\mapsto x_ix_j&\mbox{  in  } H^*(BG\times B\Z;\F_p).
\end{aligned}
\end{equation}
and then $\tilde{\tau}_g(x_ix_j)=0$. Finally for the product $y_iy_j$ we get
\begin{equation}
\begin{aligned}
(y_iy_j)&\mapsto(y_i+a_iz)(y_j+a_iz)=y_iy_j+y_ia_jz+a_izy_j&\mbox{  in  } H^*(BG\times B(\Z_p);\F_p)\\
&\mapsto y_iy_j+(a_jy_i-a_iy_j)z&\mbox{  in  } H^*(BG\times B\Z;\F_p).
\end{aligned}
\end{equation}
and then $\tilde{\tau}_g(y_iy_j)=(a_jy_i-a_iy_j)$.\\
Since we are interested in calculating the inverse transgression map for elements $\alpha\in H^4(G;\Z)$, we consider only polynomials 
of degree $4$ in \[H^*(BG;\F_p)=\F_p[x_1,\dots,x_n]\otimes\Lambda[y_1,\dots,y_n].\]Now, we present some examples of the inverse transgression map. It is easier to consider the cases $n=2$ and $n=3$. In the first case the inverse transgression map is a trivial map. In the latter the inverse transgression map is more interesting.  
\begin{example}
\begin{itemize}
\item $n=2$. For $p\neq 2$ we have that $H^*(BG;\F_p)=\F_p[x_1,x_2]\otimes\Lambda[y_1,y_2]$ with $|y_i|=1$ and $|\beta{y_i}|=|x_i|=2$. Thus, we can just consider linear combinations of the polynomials $p_1(x_1,x_2,y_1,y_2)=x_1x_2$, $p_2(x_1,x_2,y_1,y_2)=x_1y_1y_2$ and $p_3(x_1,x_2,y_1,y_2)=x_2y_1y_2$. For $p_1$ the calculations of the equation (\ref{invertrans1}) show that $\tilde{\tau}_g(p_1)=0$. Then we need to find an element $p$, which will be a ($\Z_p$)-linear combination of the polynomials $p_2$ and $p_3$, such that $\beta(p)=0$. Namely, \[\beta(p_3)=x_2(\beta(y_1)y_2-y_1\beta(y_2))=x_2(x_1y_2-y_1x_2)\]  and \[\beta(p_2)=x_1(\beta(y_1)y_2-y_1\beta(y_2))=x_1(x_1y_2-y_1x_2).\]Therefore, there does not exists such a ($\Z_p$)-linear combination. 
\item $n=3$. By analizing the degree of the polynomials, we obtain the element:
\begin{equation}\label{polinomio}
p(x_1,x_2,x_3,y_1,y_2,y_3)=x_1y_2y_3-x_2y_1y_3+x_3y_1y_2,
\end{equation}
which satisfy the condition $\beta(p)=0$. In order to check this, we use the property of $\beta$ noted in the Remark \ref{nota1}. 
\begin{equation}
\begin{aligned}
\beta(p)&=\beta(x_1y_2y_3)-\beta(x_2y_1y_3)+\beta(x_3y_1y_2)\\
&=\beta(x_1)y_2y_3+x_1\beta(y_2y_3)-\beta(x_2)y_1y_3-x_2\beta(y_1y_3)+\beta(x_3)y_1y_2+x_3\beta(y_1y_2)\\
&=x_1\beta(y_2)y_3-x_1y_2\beta(y_3)-x_2\beta(y_1)y_3+x_2y_1\beta(y_3)+x_3\beta(y_1)y_2-x_3y_1\beta(y_2)\\
&=x_1x_2y_3-x_1y_2x_3-x_2x_1y_3+x_2y_1x_3+x_3x_1y_2-x_3y_1x_2\\
&=0.\end{aligned}
\end{equation}
The inverse transgression map for an element $g=(a_1,a_2,a_3)\in (\Z_p)^3$ evaluated in the polynomial $p$ gives:
\begin{equation}\label{ejemplo}
\begin{aligned}
\tau_g(p)&=\tau_g(x_1y_2y_3)-\tau_g(x_2y_1y_3)+\tau_g(x_3y_1y_2)\\
&=x_1(a_3y_2-a_2y_3)-x_2(a_3y_1-a_1y_3)+x_3(a_2y_1-a_1y_2)\\
&=a_1(x_2y_3-x_3y_2)+a_2(x_3y_1-x_1y_3)+a_3(x_1y_2-x_2y_1)
\end{aligned}
\end{equation}
\begin{lema}\label{lemadobletrans}
Let $g=(a_1,a_2,a_3)$ and $h=(b_1,b_2,b_3)$ be elements in $G=(\Z_p)^3$. The double inverse transgression map of $p$ is equal to\[\tau_h\tau_g(p)=[(a_1,a_2,a_3)\times (b_1,b_2,b_3)]\cdot(x_1,x_2,x_3).\]
\end{lema}
\end{itemize}
\end{example}
\begin{nota}
We wish to point out that the latter example with $n=3$ shows that for $n\geq 3$ the inverse transgression map is a non-trivial map. We can always consider the element $p(x_1,\dots,x_n,y_1,\dots,y_n)=x_iy_jy_k-x_jy_iy_k+x_ky_iy_j$ to be in $H^4((Z_p)^n;\Z)$ . By similar calculations as in the equation (\ref{ejemplo}), we can prove that $\beta(p)=0$ while $\tau_g(p)\neq 0$ for $g\in (\Z_p)^n$.
\end{nota}
By using the inverse transgression map for the group $(\Z_p)^n$ presented above and by the decomposition formula presented in Theorem 3.6 in \cite{bu}, we can calculate the explicit structure of the twisted Orbifold K-theory for the Orbifold $[*/(\Z_p)^3]$ and the twist element $\alpha$ as the element in $H^3((\Z_p)^3;S^1)$ associated to the polynomial define in equation (\ref{polinomio}) via the isomorphism $H^3((\Z_p)^3;S^1)\cong H^4((\Z_p)^3;\Z)$. Note that in this case:\[{^{\alpha}}K_{orb}([*/(\Z_p)^3])=\bigoplus_{g\in(\Z_p)^3}{^{\alpha_g}}K_{(\Z_p)^3}(*)\cong\bigoplus_{g\in(\Z_p)^3}R_{\alpha_g}((\Z_p)^3).\]Now, for each $g\in \Z_p$, the decomposition formula implies:
\[R_{\alpha_g}((\Z_p)^3)\otimes \Q\cong\prod_{g,h \in (\Z_p)^3} \left(\Q(\zeta_{p})_{h,\alpha_g}\right)^{(\Z_p)^3},\]where $\zeta_p$ is a $p$-root of the unity. Note that the action of $(\Z_p)^3$ on $\Q(\zeta_{p})_{h,\alpha_g}$ is to multiply by the double inverse transgression map evaluated on $k\in (\Z_p)^3$, $\tau_h(\alpha_g)(k)$ and by the Lemma (\ref{lemadobletrans}) we get
\begin{equation}\left(\Q(\zeta_{p})_{h,\alpha_g}\right)^{(\Z_p)^3}=\begin{cases} \Q(\zeta_p) & \text{if $g=\lambda h$, $\lambda \in \Z_p$,}
\\
0 &\text{else.}
\end{cases}\end{equation}So, for $h\neq 0$, we have that\[R_{\alpha_g}((\Z_p)^3)\otimes \Q=\prod_{\lambda\in \Z_p}{\Q(\zeta_p)},\]while for $g=0$, we get
\[R_{\alpha_1}((\Z_p)^3)\otimes \Q=\prod_{\lambda\in (\Z_p)^3}{\Q(\zeta_p)}.\]
Then, the module the twisted Orbifold K-theory for the Orbifold $[*/(\Z_p)^3]$ turns out to be:
\[{^{\alpha}}K_{orb}([*/(\Z_p)^3])\otimes \Q=\prod_{\lambda\in \Z_p}{\Q(\zeta_p)}\bigoplus \prod_{\lambda\in (\Z_p)^3}{\Q(\zeta_p)} \]and the product structure is defined via the product of the elements in $\Q(\zeta_p)$.
\section{Final remarks}
With the result presented in Section 4 about the Grothendieck ring associated to the semi-group of representations of the twisted Drinfeld double $\D$ and the twisted Orbifold K-theory, we found a nice relation between two structures coming from different sources. As we already said, the Orbifold $[*/G]$ is a particular case of a more general kind of Orbifolds obtained by the almost free action of a compact Lie group $G$ on a compact manifold $M$. With a little more structure, the stringy product introduced in Section 3 can be extended to a stringy product on the module ${^{\alpha}}K_{orb}([M/G])$ (in the same way as in \cite{bu}), where $[M/G]$ denotes the Orbifold structure obtained by the almost free action (see \cite{AR} for the details of this structure). Therefore, under suitable hypoteses we can think about the twisted Orbifold K-theory ${^{\omega}}K_{orb}([M/G])$ as a more general object which coincides with the Grothendieck ring $R(\D)$ if $G$ is a finite group and $M=\{*\}$. Nevertheless, we shall explore the interpretation and consequences of this more general object. Next, we focus our atention on the 
results obtained in Section 5, where we establish an explicit relation between the twisted Orbifold K-theories of the Orbifolds $[*/H]$ and $[*/(\Z_p)^n]$, where $H$ is a particular extra special $p$-group. In the same spirit, we look for some general relation between the twisted Orbifold K-theories ${^{\alpha}}K_{orb}([M/G])$ and ${^{\beta}}K_{orb}([M/K])$ of the Orbifolds $[M/G]$ and $[M/K]$, for suitable twistings $\alpha\in H^3(G;S^1)$ and $\beta\in H^3(H;S^1)$, and appropriate actions of the finite groups $G$ and $K$ on a compact manifold $M$. In the same way, we hope that some analogous results may be obtained if $G$ and $K$ are compact Lie groups acting almost freely on a compact manifold $M$. By our preliminary observations, in order to obtain such results, some hypothesis on the almost free actions of the compact Lie groups $G$ and $K$ must be added.
\bibliographystyle{amsplain}
\bibliography{biblio3}
\end{document}